\documentclass[12pt]{amsart}

\usepackage{amscd,amsmath,amssymb,amsthm}
\usepackage{mathtools}
\usepackage[colorlinks=true]{hyperref}
\mathtoolsset{showonlyrefs}
\usepackage{tikz-cd}
\usepackage{fullpage}

\numberwithin{equation}{section}

\theoremstyle{plain}
\newtheorem{proposition}{Proposition}[section]
\newtheorem{theorem}[proposition]{Theorem}

\theoremstyle{definition}

\newtheorem{remark}[proposition]{Remark}
\newtheorem{example}[proposition]{Example}

\title[Gauss--Manin connection for abelian schemes]{A note on the Gauss--Manin connection  for abelian schemes}

\author{Tiago J. Fonseca and Nils Matthes}

\address{IMECC - Unicamp,
	Rua Ségio Buarque de Holanda, 651 Cidade Universitária,
	CEP 13083-859, 
	Campinas, SP, Brazil}
\email{tjfonseca@ime.unicamp.br}

\address{Department of Mathematical Sciences, 
	University of Copenhagen, 
	Universitetsparken 5,
	2100 Copenhagen Ø,
	Denmark}
\email{nils.oliver.matthes@gmail.com}

\begin{document}
	
	\begin{abstract}
		We study differential forms on the universal vector extension $A^\natural$ of an abelian scheme $A$ in characteristic zero, and derive a new construction of the $D$-group scheme structure on $A^\natural$. This gives, in particular, a rather simple description of the Gauss--Manin connection on the de Rham cohomology of $A$ in terms of global algebraic differential forms on $A^\natural$. The key ingredient is the computation of the coherent cohomology of $A^{\natural}$, due to Coleman and Laumon.
		
	\end{abstract}
	
	\maketitle
	
	\section{Introduction}
	
	Let $k$ be a field. It is well known that the de Rham cohomology groups of a smooth morphism $f: X\rightarrow S$ of smooth $k$-schemes are equipped with an integrable $k$-connection, the \emph{Gauss--Manin connection} \cite{Grothendieck:DeRham,KatzOda}. For an abelian scheme $f:A \to S$, Grothendieck explained, in a famous letter to Tate, that the Gauss--Manin connection is related to the `crystalline nature' of the universal vector extension $A^\natural$ of $A$ via the natural isomorphism $\operatorname{Lie}_SA^{\natural} \cong H^1_{\rm dR}(A/S)^{\vee}$ (see \cite{MazurMessing}). When $\operatorname{char}(k)=0$, this `crystalline nature' amounts to a $D$-group scheme structure on $A^\natural/S$, an algebraic analogue of the differential-geometric notion of an integrable Ehresmann connection (see \cite{Buium} and \cite[Section 6]{Bost}).
	
	In the case where $k=\mathbb C$, we may describe this $D$-group scheme structure on $A^{\natural}/S$ in terms of the uniformization $\exp: V^{\rm an}\to A^{\natural,\rm an}$ of the analytification of $A^{\natural}$, where $V$ is the vector group $\mathbb V( (\operatorname{Lie}_SA^{\natural})^{\vee}) \cong \mathbb{V}(H^1_{\rm dR}(A/S))$. Namely, the Gauss--Manin connection on $H^1_{\rm dR}(A/S)$ equips $V$ with a natural structure of a `linear' $D$-group scheme, which descends to an \emph{analytic} $D$-group scheme structure on $A^{\natural,\rm an}$ via $\exp$. One can then show that the latter arises as the analytification of an \emph{algebraic} $D$-group scheme structure on $A^\natural$. From this point of view, the algebraicity of the $D$-group scheme structure is rather surprising, given that GAGA fails for $A^{\natural}$ (cf. \cite[2.3.1]{Bost}).
	
	In this note, we give a direct construction of the $D$-group scheme structure on $A^\natural/S$ in characteristic zero which, to the best of our knowledge, has not appeared in the literature so far. On the one hand, this uses the computation of the coherent cohomology of the structure map $g: A^\natural\rightarrow S$, independently obtained by Coleman \cite[Corollary 2.7]{Coleman} and Laumon \cite[Th\'eor\`eme 2.4.1]{Laumon}. On the other hand, it also requires a detailed study of differential forms, both `relative' and `absolute', on the universal vector extension, which is carried out in Section \ref{ssec:differentialUVE}, and which may be of independent interest. Using these two ingredients, we then construct a $D$-group scheme structure on $A^\natural$ (Theorem \ref{thm:Dgroupscheme}) by means of the canonical retraction $\rho: g_*\Omega^1_{A^\natural/k}\rightarrow \Omega^1_{S/k}$ given by pullback along the zero section $e\in A^\natural(S)$ (Theorem \ref{thm:canonicalsplitting}). This leads to a particularly simple description of the Gauss--Manin connection on $H^1_{\rm dR}(A/S)$ (Proposition \ref{prop:UVEGM}). Finally, we show that the $D$-group scheme structure on $A^\natural$ described above agrees with the one coming from its `crystalline nature' (Theorem \ref{thm:algebraic}).

	\subsection*{Acknowledgements}
	
	We are grateful to Netan Dogra for pointing out reference \cite{Buium}. This project has received funding from the European Research Council (ERC) under the European Union’s Horizon 2020 research and innovation programme (grant agreement No. 724638). The first author is currently supported by the grant $\#$2020/15804-1, São Paulo Research Foundation (FAPESP), and the second author is currently a Walter Benjamin Fellow of the Deutsche Forschungsgemeinschaft (DFG).
	
	\section{Universal vector extensions and de Rham cohomology of abelian schemes} \label{sec:UVEdeRham}
	
	\subsection{Review of de Rham cohomology} \label{ssec:deRham}
	
	Let $f: X\rightarrow S$ be a morphism of schemes. The $q$-th de Rham cohomology sheaf of $f$ is the $\mathcal{O}_S$-module $H^q_{\rm dR}(X/S)\coloneqq R^qf_*(\Omega^\bullet_{X/S})$, where $Rf_*: D^+(f^{-1}\mathcal{O}_S) \to D^+(\mathcal{O}_S)$ is the right derived functor of $f_*$. If $f$ is quasi-compact and quasi-separated, then $H^q_{\rm dR}(X/S)$ is a quasi-coherent $\mathcal{O}_S$-module \cite[\href{https://stacks.math.columbia.edu/tag/0FLX}{Lemma 0FLX}]{Stacks}.
	
	Now let $T$ denote an arbitrary scheme, and let $f: X\rightarrow S$ be a smooth morphism of finite presentation of smooth $T$-schemes. Then $H^q_{\rm dR}(X/S)$ is equipped with an integrable $T$-connection
	\[
	\nabla^q: H^q_{\rm dR}(X/S)\longrightarrow \Omega^1_{S/T}\otimes H^q_{\rm dR}(X/S),
	\]
	the \emph{Gauss--Manin connection} \cite{Grothendieck:DeRham,KatzOda}, which is constructed as follows. Consider the filtered complex $(\Omega^\bullet_{X/T},\{F^p\}_{p\geq 0})$, where
	\[
	F^p\coloneqq \operatorname{im}(f^\ast\Omega^p_{S/T}\otimes \Omega^\bullet_{X/T}[-p]\stackrel{\wedge}{\longrightarrow} \Omega^\bullet_{X/T}).
	\]
	By smoothness, its graded pieces are $F^p/F^{p+1}\cong f^\ast\Omega^p_{S/T}\otimes \Omega^\bullet_{X/S}[-p]$. Hence, the first page of the corresponding spectral sequence gives rise to a morphism
	\[
	d_1^{0,q}: R^qf_*(\Omega^\bullet_{X/S}) \longrightarrow R^{q+1}f_*(f^\ast\Omega^1_{S/T} \otimes \Omega^\bullet_{X/S}[-1])\cong \Omega^1_{S/T}\otimes R^qf_*(\Omega^\bullet_{X/S}),
	\]
	which can be shown to be an integrable $T$-connection. By definition $\nabla^q \coloneqq d_1^{0,q}$ is the Gauss--Manin connection.
	
	\subsection{The universal vector extension of an abelian scheme} \label{ssec:UVE}
	
	Let $f: A\rightarrow S$	be an \emph{abelian scheme}, i.e., $f$ is a proper smooth $S$-group scheme (necessarily commutative), with geometrically connected fibres. Its \emph{universal vector extension} \cite{MazurMessing} is a commutative $S$-group scheme $g: A^\natural\rightarrow S$ which fits into a short exact sequence (of fppf abelian sheaves)
	\begin{equation} \label{eqn:UVE}
		\begin{tikzcd}
			0\arrow{r}&\mathbb V(R^1f_*\mathcal{O}_A)\arrow{r}&A^\natural\arrow{r}{\pi}&A\arrow{r}&0
		\end{tikzcd}
	\end{equation}
	satisfying the following universal property: given a quasi-coherent $\mathcal{O}_S$-module $\mathcal{M}$, the morphism of abelian groups
	\[
	\begin{aligned}
		\operatorname{Hom}_{\mathcal{O}_S}(R^1f_*\mathcal{O}_A,\mathcal{M}) &\longrightarrow \operatorname{Ext}^1_{S_{\rm fppf}}(A,\mathbb V(\mathcal{M}))\\
		\varphi&\longmapsto \mbox{``the class of the pushout of \eqref{eqn:UVE} along $\mathbb V(\varphi)$''}
	\end{aligned}
	\]
	is an isomorphism \cite[Proposition I.1.10]{MazurMessing}.
	
	It follows from \eqref{eqn:UVE} that $\pi: A^\natural\rightarrow A$ is an fppf-torsor under the vector group $\mathbb V(f^\ast R^1f_*\mathcal{O}_A)$, hence an affine bundle, since $R^1f_*\mathcal{O}_A$ is a locally free $\mathcal{O}_S$-module of finite rank. In particular, $g$ is smooth, separated, of finite presentation, and has geometrically connected fibres.\footnote{We warn the reader that $g$ is neither proper nor affine!} Moreover, the formation of the universal vector extension commutes with base change in the following sense: given a morphism of schemes $S'\rightarrow S$, there is a canonical isomorphism of $S'$-group schemes $A^\natural\times_S S' \stackrel{\sim}{\rightarrow} (A\times_S S')^{\natural}$. This follows from the interpretation of $g: A^\natural\rightarrow S$ as a moduli scheme of line bundles with integrable connection \cite[I.2.6, I.3.2, I.4.2]{MazurMessing}.
	
	Now assume that $S$ is locally of finite type over $\operatorname{Spec} \mathbb C$. Then the analytification $A^{\natural,\rm an}$ of $A^\natural$ is uniformized by the vector $S$-group scheme $V\coloneqq \mathbb V(H^1_{\rm dR}(A/S))$. More precisely, there is a short exact sequence of commutative complex Lie groups over $S^{\rm an}$
	\begin{equation} \label{eqn:uniformization}
		\begin{tikzcd}
			0\arrow{r}&L\arrow{r}&V^{\rm an}\arrow{r}{\exp}&A^{\natural,\rm an}\arrow{r}&0,
		\end{tikzcd}
	\end{equation}
	where $L$ is the espace étalé associated to $(R^1f_*^{\rm an}\mathbb Z)^\vee$, and the map $L\to V^{\rm an}$ is induced by the morphism $(R^1f_*^{\rm an}\mathbb Z)^\vee \to (H^1_{\rm dR}(A/S)^{\rm an})^\vee$ sending a locally constant family of topological $1$-cycles $\gamma$ to the integration functional $\alpha \mapsto \int_\gamma\alpha $ (cf. \cite[I.4.4]{MazurMessing}).
	
	\subsection{Coherent cohomology of the universal vector extension} \label{ssec:CoherentUVE}
	
	If $S$ has characteristic zero, then the coherent cohomology of $g$ is particularly simple.
	
	\begin{theorem}[Coleman, Laumon] \label{thm:Laumon}
		If $\operatorname{char}(S)=0$, then the adjunction $\mathcal{O}_S\rightarrow Rg_*\mathcal{O}_{A^\natural}$ is an isomorphism in $D^{\rm b}_{\rm Qcoh}(\mathcal{O}_S)$. In particular, $Rg_*\mathcal{O}_{A^\natural}$ is a perfect object of $D^{\rm b}_{\rm Qcoh}(\mathcal{O}_S)$, and its formation commutes with arbitrary change of base.
	\end{theorem}
	
	\begin{proof}
		In the case where $S$ is locally Noetherian, the first assertion is precisely \cite[Th\'eor\`eme 2.4.1]{Laumon}; the general case may be deduced from the locally Noetherian case by a standard approximation argument. Alternatively, see \cite[Corollary 2.7]{Coleman}.
		
		The second assertion is an immediate consequence of the first, using that the formation of the universal vector extension commutes with base change.
	\end{proof}
	
	\begin{remark}
		In Theorem \ref{thm:Laumon}, the assertion that the canonical map $\mathcal{O}_S\rightarrow g_*\mathcal{O}_{A^\natural}$ is an isomorphism holds more generally when $S$ is flat over $\operatorname{Spec}\mathbb Z$ \cite[Corollary 2.4]{Coleman}.
	\end{remark}
	
	\begin{remark}
		By \cite[Remark 2.4]{Brion}, Theorem \ref{thm:Laumon} is false if the characteristic of $S$ is positive.
	\end{remark}
	
	\subsection{Differential forms on the universal vector extension} \label{ssec:differentialUVE}
	
	We now apply Theorem \ref{thm:Laumon} to the study of sheaves of differential forms on the universal vector extension of an abelian scheme. 
	
	We begin with relative differential forms. Denote by $e\in A^\natural(S)$ the zero section of $A^\natural$. For every $q\geq 0$, there is a canonical isomorphism
	\begin{equation} \label{eqn:invariant}
		g^\ast e^\ast \Omega^q_{A^\natural/S}\stackrel{\sim}\longrightarrow \Omega^q_{A^\natural/S}
	\end{equation}
	given by extending sections of $e^\ast\Omega^q_{A^\natural/S}$ to invariant differential forms in $\Omega^q_{A^\natural/S}$ via the group law \cite[4.2, Proposition 2]{BLR}. Applying the functor $Rg_*$, we obtain a natural isomorphism
	\[
	Rg_*(g^\ast e^\ast\Omega^q_{A^\natural/S}) \stackrel{\sim}\longrightarrow Rg_*\Omega^q_{A^\natural/S}
	\]
	in $D^{\rm b}_{\rm Qcoh}(\mathcal{O}_S)$.
	
	\begin{proposition} \label{prop:relativeformsUVE}
		If $\operatorname{char}(S)=0$, then the adjunction
		\begin{equation} \label{eqn:adjunction}
			e^\ast\Omega^q_{A^\natural/S} \longrightarrow Rg_*(g^\ast e^\ast\Omega^q_{A^\natural/S}) \cong Rg_*\Omega^q_{A^\natural/S}
		\end{equation}
		is an isomorphism in $D^{\rm b}_{\rm Qcoh}(\mathcal{O}_S)$. In particular:
		\begin{itemize}
			\item[(i)] The $\mathcal{O}_S$-module $g_*\Omega^q_{A^\natural/S}$ is locally free of finite rank, and its formation commutes with arbitrary change of base.
			\item[(ii)]
			The natural map $g^\ast g_*\Omega^q_{A^\natural/S}\rightarrow \Omega^q_{A^\natural/S}$ is an isomorphism of $\mathcal{O}_{A^{\natural}}$-modules.
			\item[(iii)] The natural map $\bigwedge^q g_*\Omega^1_{A^\natural/S} \rightarrow g_*\Omega^q_{A^\natural/S}$ is an isomorphism of $\mathcal{O}_S$-modules.
			\item[(iv)] Every section of $g_*\Omega^q_{A^\natural/S}$ is a closed differential form. In other words, the relative differential $d_{A^\natural/S}$ vanishes identically on $g_*\Omega^q_{A^\natural/S}$.
		\end{itemize}
	\end{proposition}
	
	\begin{proof}
		That \eqref{eqn:adjunction} is an isomorphism in $D^{\rm b}_{\rm Qcoh}(\mathcal{O}_S)$ follows immediately from Theorem \ref{thm:Laumon} using the projection formula. The remaining assertions are derived from the isomorphism $g_*\Omega^q_{A^\natural/S}\cong e^\ast\Omega^q_{A^\natural/S}$ as follows. In the first claim of assertion (i), we use that $\Omega^q_{A^\natural/S}$ is locally free of finite rank; in the second, that the universal vector extension commutes with base change. Assertion (ii) is a restatement of \eqref{eqn:invariant}. Similarly, assertion (iii) follows from the fact that $e^\ast$ commutes with the wedge product. Finally, assertion (iv) is a general property of invariant differential forms on commutative smooth group schemes (cf. \cite[Ch. 3, \S 3.14, Proposition 51]{Bourbaki} or \cite[Lemma 2.1]{Coleman}).
	\end{proof}
	
	We next study `absolute' differential forms on $A^\natural$. For this, let $T$ be an arbitrary scheme of characteristic zero, and assume that $S$ is a smooth $T$-scheme. Then there is a short exact sequence
	\begin{equation} \label{eqn:exactsequencedifferentialsUVE}
		\begin{tikzcd}
			0\arrow{r}& g^\ast\Omega^1_{S/T}\arrow{r}&\Omega^1_{A^\natural/T}\arrow{r}&\Omega^1_{A^\natural/S}\arrow{r}&0
		\end{tikzcd}
	\end{equation}
	of locally free $\mathcal{O}_{A^\natural}$-modules of finite rank.
	
	\begin{proposition} \label{prop:absoluteformsUVE}
		With notation and conventions as above, the following are true:
		\begin{itemize}
			\item[(i)] The pushforward of \eqref{eqn:exactsequencedifferentialsUVE} along $g$ gives a short exact sequence of $\mathcal{O}_S$-modules
			\begin{equation} \label{eqn:exactsequencedifferentialspushforward}
				\begin{tikzcd}
					0\arrow{r}& \Omega^1_{S/T}\arrow{r}&g_*\Omega^1_{A^\natural/T}\arrow{r}&g_*\Omega^1_{A^\natural/S}\arrow{r}&0.
				\end{tikzcd}
			\end{equation}
			\item[(ii)] The $\mathcal{O}_S$-module $g_*\Omega^1_{A^\natural/T}$ is locally free of finite rank, and its formation commutes with arbitrary change of base.
			\item[(iii)] The natural map $g^\ast g_*\Omega^1_{A^\natural/T}\rightarrow \Omega^1_{A^\natural/T}$ is an isomorphism of $\mathcal{O}_{A^{\natural}}$-modules.
			\item[(iv)] For every $q\geq 0$, the natural map $\bigwedge^q g_*\Omega^1_{A^\natural/T} \rightarrow g_*\Omega^q_{A^\natural/T}$ is an isomorphism of $\mathcal{O}_S$-modules.
		\end{itemize}
	\end{proposition}
	
	\begin{proof}
		Assertion (i) is an immediate consequence of Theorem \ref{thm:Laumon} using the projection formula and the long exact sequence in cohomology. Assertion (ii) follows directly from (i) in concert with Proposition \ref{prop:relativeformsUVE}.(i). In order to prove assertion (iii), consider the commutative diagram
		\[
		\begin{tikzcd}
			0\arrow{r}&g^\ast\Omega^1_{S/T}\arrow{r}\arrow[equal]{d}&g^\ast g_*\Omega^1_{A^\natural/T}\arrow{r}\arrow{d}&g^*g_*\Omega^1_{A^\natural/S}\arrow{r}\arrow{d}&0\\
			0\arrow{r}&g^\ast\Omega^1_{S/T}\arrow{r}&\Omega^1_{A^\natural/T}\arrow{r}&\Omega^1_{A^\natural/S}\arrow{r}&0,
		\end{tikzcd}
		\]
		where exactness of the top row follows from (i) and from exactness of $g^\ast$. Both the left-hand arrow and the right-hand arrow are isomorphisms (the latter by Proposition \ref{prop:relativeformsUVE}.(ii)), hence so is the middle one, by the five-lemma. Finally, assertion (iv) is local on $S$, hence, by local freeness, we may assume that the exact sequence \eqref{eqn:exactsequencedifferentialsUVE} splits: $\Omega^1_{A^\natural/T}\cong g^\ast\Omega^1_{S/T}\oplus \Omega^1_{A^\natural/S}$. Thus we get isomorphisms
		\begin{align}
			\bigwedge^q g_*\Omega^1_{A^\natural/T} &\cong \bigwedge^q (\Omega^1_{S/T}\oplus g_*\Omega^1_{A^\natural/S}) & \mbox{(projection formula)}\\
			&\cong \bigoplus_{i+j=q} \bigwedge^i \Omega^1_{S/T}\otimes \bigwedge^j g_*\Omega^1_{A^\natural/S} & \\
			&\cong \bigoplus_{i+j=q} \Omega^i_{S/T}\otimes g_*\Omega^j_{A^\natural/S} & \mbox{(Proposition \ref{prop:relativeformsUVE}.(iii))}\\
			&\cong g_*\left( \bigoplus_{i+j=q} g^\ast\Omega^i_{S/T}\otimes \Omega^j_{A^\natural/S} \right) & \mbox{(projection formula)}\\
			&\cong g_*\left( \bigwedge^q (g^\ast\Omega^1_{S/T}\oplus \Omega^1_{A^\natural/S}) \right)& \\
			&\cong g_*\Omega^q_{A^\natural/T}, &
		\end{align}
		as desired.
	\end{proof}
	
	Now, recall from Section \ref{ssec:deRham} the definition of the filtration $\{F^p\}_{p\geq 0}$ on $\Omega^\bullet_{A^\natural/T}$. It gives rise to a filtration by subcomplexes $\{g_*F^p\}_{p\geq 0}$ on $g_*\Omega^\bullet_{A^\natural/T}$.
	
	\begin{proposition} \label{prop:filtrationpushforward}
		The following assertions hold for all $p\geq 0$:
		\begin{itemize}
			\item[(i)]
			The canonical map $g_*F^p/g_*F^{p+1} \rightarrow g_*\left( F^p/F^{p+1} \right)$ is an isomorphism. In particular, we have
			$
			g_*F^p/g_*F^{p+1} \cong \Omega^p_{S/T}\otimes g_*\Omega^\bullet_{A^\natural/S}[-p].
			$
			\item[(ii)]
			We have an equality $g_*F^p=\operatorname{im}\left(\Omega^p_{S/T}\otimes g_*\Omega^\bullet_{A^\natural/T}[-p] \rightarrow g_*\Omega^\bullet_{A^\natural/T}\right)$ of subcomplexes of $g_*\Omega^\bullet_{A^\natural/T}$.
		\end{itemize}
	\end{proposition}
	
	\begin{proof}
		As the statement is local on $S$, we may assume that $S\rightarrow T$ is of finite presentation, which implies in particular that $F^p=0$ for $p\gg 0$. To prove assertion (i), consider the short exact sequence of complexes
		\[
		\begin{tikzcd}
			0\arrow{r}&F^{p+1}\arrow{r}&F^p\arrow{r}&F^p/F^{p+1}\arrow{r}&0.
		\end{tikzcd}
		\]
		Theorem \ref{thm:Laumon} and the projection formula imply that each term of $F^p/F^{p+1}\cong g^\ast\Omega^p_{S/T}\otimes \Omega^\bullet_{A^\natural/S}[-p]$ is $g_*$-acyclic. Thus, by descending induction on $p$ (which is possible as the filtration $\{F^p\}_{p\geq 0}$ is finite) and by the long exact sequence in cohomology, pushforward along $g$ yields a short exact sequence of complexes
		\[
		\begin{tikzcd}
			0\arrow{r}&g_*F^{p+1}\arrow{r}&g_*F^p\arrow{r}&g_*\left(F^p/F^{p+1}\right)\arrow{r}&0,
		\end{tikzcd}
		\]
		from which assertion (i) follows immediately.
		
		Now, for assertion (ii), consider the natural map given by adjunction
		\begin{equation} \label{eqn:adjunctioncomplexes}
			\Omega^p_{S/T}\otimes g_*\Omega^\bullet_{A^\natural/T}[-p] \longrightarrow g_*(g^\ast\Omega^p_{S/T}\otimes \Omega^\bullet_{A^\natural/T}[-p])\longrightarrow g_*F^p,
		\end{equation}
		and set $G^p=\operatorname{im}\left(\Omega^p_{S/T}\otimes g_*\Omega^\bullet_{A^\natural/T}[-p] \rightarrow g_*\Omega^\bullet_{A^\natural/T}\right)$. Since $g_*F^p$ is a subcomplex of $g_*F^0=G^0$, the universal property of images implies that \eqref{eqn:adjunctioncomplexes} factors through a map $G^p\rightarrow g_*F^p$. We thus obtain a commutative diagram
		\begin{equation} \label{eqn:diagrampushforwardfiltration}
			\begin{tikzcd}
				0\arrow{r}& G^{p+1}\arrow{r}\arrow{d}&G^p\arrow{r}\arrow{d}&G^p/G^{p+1}\arrow{r}\arrow{d}&0\\
				0\arrow{r}&g_*F^{p+1}\arrow{r}&g_*F^p\arrow{r}&g_*F^p/g_*F^{p+1}\arrow{r}&0.
			\end{tikzcd}
		\end{equation}
		On the other hand, it follows from Proposition \ref{prop:relativeformsUVE}.(iii), Proposition \ref{prop:absoluteformsUVE}.(iv), and the fact that \eqref{eqn:exactsequencedifferentialspushforward} is locally split that $G^p/G^{p+1} \cong \Omega^p_{S/T}\otimes g_*\Omega^\bullet_{A^\natural/S}[-p]$. Hence, the right-hand arrow in \eqref{eqn:diagrampushforwardfiltration} is an isomorphism, and by descending induction on $p$ together with the five-lemma we conclude that $G^p\rightarrow g_*F^p$ is an isomorphism. This proves assertion (ii).
	\end{proof}

	\subsection{Universal vector extensions and de Rham cohomology}
	
	The de Rham cohomology of $A/S$ can be described in terms of global differentials on $A^\natural/S$ as follows. 
	\begin{proposition}[cf. {\cite[Theorem 2.2]{Coleman}}] \label{prop:UVEdeRham}
		Let $q\geq 0$. If $\operatorname{char}(S)=0$, then there are canonical isomorphisms of $\mathcal{O}_S$-modules
		\begin{equation} \label{eqn:UVEdeRham}
			g_*\Omega^q_{A^\natural/S}\stackrel{\sim}\longrightarrow H^q_{\rm dR}(A^\natural/S) \stackrel{\sim}\longleftarrow H^q_{\rm dR}(A/S),
		\end{equation}
		where the right-hand arrow is induced by $\pi: A^\natural\rightarrow A$ \eqref{eqn:UVE}.
	\end{proposition}
	
	\begin{proof}
		Since ${\rm char}(S) = 0$ and $\pi$ is an affine bundle, it follows from the K\"unneth formula that the right-hand arrow is an isomorphism. For the left-hand arrow, we note that Proposition \ref{prop:relativeformsUVE} implies that $\Omega^q_{A^{\natural}/S}$ is $g_*$-acylic for every $q\ge 0$. Thus, we obtain canonical isomorphisms $H^q(g_*\Omega^\bullet_{A^\natural/S})\cong H^q_{\rm dR}(A^\natural/S)$. On the other hand, by Proposition \ref{prop:relativeformsUVE}.(iv), we have $H^q(g_*\Omega^\bullet_{A^\natural/S})=g_*\Omega^q_{A^\natural/S}$, ending the proof.
	\end{proof}
	
	\begin{remark}
		Without any assumption on the characteristic of $S$, one can show that $H^q_{\rm dR}(A/S)\cong e^\ast\Omega^q_{A^\natural/S}$, for all $q\geq 0$ \cite[4.1.7]{MazurMessing}.
	\end{remark}
	
	\section{\texorpdfstring{$D$}{}-group scheme structure on the universal vector extension}
	
	Throughout this section, let $T$ be an arbitrary scheme of characteristic zero.
	
	\subsection{Review of \texorpdfstring{$D$}{}-group schemes}
	
	Let $S$ be a smooth $T$-scheme. Recall that a $D$\emph{-scheme} over $S$ (cf. \cite[6.1]{Bost}) is a pair $(X,\mathcal{F})$ consisting of a smooth $S$-scheme $g: X\rightarrow S$, and an \emph{integrable} $\mathcal{O}_X$-submodule $\mathcal{F}\hookrightarrow \mathcal{T}_{X/T}$ (i.e., $\mathcal{F}$ is closed under the Lie bracket of vector fields) which splits the exact sequence
	\[
	\begin{tikzcd}
		0\arrow{r}&\mathcal{T}_{X/S}\arrow{r}&\mathcal{T}_{X/T}\arrow{r}{Dg}&g^\ast\mathcal{T}_{S/T}\arrow{r}&0.
	\end{tikzcd}
	\]
	A \emph{morphism of $D$-schemes} $(X_1,\mathcal{F}_1)\rightarrow (X_2,\mathcal{F}_2)$ is a morphism of $S$-schemes $\phi: X_1\rightarrow X_2$ whose `absolute' differential $D\phi: \mathcal{T}_{X_1/T}\rightarrow \phi^*\mathcal{T}_{X_2/T}$ maps $\mathcal{F}_1$ into $\phi^*\mathcal{F}_2$. The category of $D$-schemes admits finite products \cite[6.1]{Bost}, and a \emph{$D$-group scheme} is defined as a group object in the category of $D$-schemes.
	
	In this paper, the dual point of view is more convenient. To give $\mathcal{F}$ as above is equivalent to giving an \emph{integrable} $\mathcal{O}_X$-submodule $\mathcal{G}\hookrightarrow \Omega^1_{X/T}$ (i.e., $d\mathcal{G}\subset \operatorname{im}(\mathcal{G}\otimes \Omega^1_{X/T}\stackrel{\wedge}{\rightarrow} \Omega^2_{X/T})$) which splits the dual exact sequence:
	\begin{equation} \label{eqn:exactsequencedifferentials}
		\begin{tikzcd}
			0\arrow{r}&g^\ast\Omega^1_{S/T}\arrow{r}&\Omega^1_{X/T}\arrow{r}&\Omega^1_{X/S}\arrow{r}&0.
		\end{tikzcd}
	\end{equation}
	The equivalence is given explicitly by setting $\mathcal{F}=(\Omega^1_{X/T}/\mathcal{G})^\vee$ (cf. \cite[Ch. II, \S 2.4]{HectorHirsch}). Then, a morphism of $S$-schemes $\phi: X_1 \to X_2$ is a morphism of $D$-schemes if the pullback map $\phi^*\Omega^1_{X_2/T} \to \Omega^1_{X_1/T}$ sends $\phi^*\mathcal{G}_2$ to $\mathcal{G}_1$.
	
	\begin{example}[Linear $D$-group schemes]\label{ex:linear-d-group-scheme}
		For later reference, let us recall how an integrable $T$-connection $\nabla : \mathcal{E} \to \Omega^1_{S/T}\otimes \mathcal{E}$ on a locally free $\mathcal{O}_S$-module of finite rank $\mathcal{E}$ defines a $D$-group scheme structure on the vector group $p:\mathbb{V}(\mathcal{E}) \to S$. By adjunction, the inclusion $\mathcal{E} \hookrightarrow \mathrm{Sym}(\mathcal{E}) \cong p_*\mathcal{O}_{\mathbb{V}(\mathcal{E})}$ yields a morphism $p^*\mathcal{E} \hookrightarrow \mathcal{O}_{\mathbb{V}(\mathcal{E}})$. Thus, we can regard the pullback of $\nabla$ as a map 
		\[
		p^*\nabla : p^*\mathcal{E} \rightarrow \Omega^1_{\mathbb{V}(\mathcal{E})/T}\otimes p^*\mathcal{E} \hookrightarrow \Omega^1_{\mathbb{V}(\mathcal{E})/T}.
		\]
		On the other hand, we can also consider the exterior derivative $d : \mathcal{O}_{\mathbb{V}(\mathcal{E})} \to \Omega^1_{\mathbb{V}(\mathcal{E})/T}$ and restrict it to $p^*\mathcal{E}$. The difference of these two maps is an $\mathcal{O}_{\mathbb{V}(\mathcal{E})}$-linear morphism
		\[
		\sigma_{\nabla} = d - p^*\nabla : p^*\mathcal{E} \to \Omega^1_{\mathbb{V}(\mathcal{E})/T}
		\]
		which splits the exact sequence
		\begin{equation}
			\begin{tikzcd}
				0\arrow{r}&p^\ast\Omega^1_{S/T}\arrow{r}&\Omega^1_{\mathbb{V}(\mathcal{E})/T}\arrow{r}&\Omega^1_{\mathbb{V}(\mathcal{E})/S}\arrow{r}&0
			\end{tikzcd}
		\end{equation}
		under the canonical isomorphism $\Omega^1_{\mathbb{V}(\mathcal{E})/S}\cong p^*\mathcal{E}$. Explicitly, if $(x_1,\ldots,x_r)$ is a local framing of $\mathcal{E}$ and $\nabla x_j = \sum_{i=1}^r\alpha_{ij}\otimes x_i$, then
		\[
		\sigma_{\nabla} (dx_j) = dx_j - \sum_{i=1}^r x_i \alpha_{ij}.
		\]
		One easily checks that $(\mathbb{V}(\mathcal{E}),\mathrm{im}(\sigma_{\nabla}))$ is a $D$-group scheme over $S$.
	\end{example}
	
	\subsection{The case of the universal vector extension}
	
	Let $g: A^\natural\rightarrow S$ be the universal vector extension of an abelian scheme $f: A\rightarrow S$, and denote by $e \in A^{\natural}(S)$ the zero section.
	
	\begin{theorem} \label{thm:canonicalsplitting}
		With the above notation:
		\begin{itemize}
			\item[(i)]
			Let $e^\ast: \Omega^1_{A^\natural/T}\rightarrow e_*\Omega^1_{S/T}$ be the morphism given by pullback along the zero section. The map $\rho \coloneqq g_*(e^\ast): g_*\Omega^1_{A^\natural/T}\rightarrow \Omega^1_{S/T}$ is a retraction of \eqref{eqn:exactsequencedifferentialspushforward}.
			\item[(ii)]
			Let $\mathcal{N}\coloneqq \ker(\rho) \hookrightarrow g_*\Omega^1_{A^\natural/T}$. We have
			\begin{equation} \label{eqn:strongerintegrability}
				d\mathcal{N} \subset \operatorname{im}(\Omega^1_{S/T}\otimes \mathcal{N}\stackrel{\wedge}\longrightarrow g_*\Omega^2_{A^\natural/T}),
			\end{equation}
			where $d=d_{A^\natural/T}$.
		\end{itemize}
	\end{theorem}
	
	\begin{proof}
		Assertion (i) follows immediately from the fact that $e$ is a section of $g$. For assertion (ii), consider the decomposition
		\[
		g_*\Omega^2_{A^\natural/T}\cong \Omega^2_{S/T}\oplus (\Omega^1_{S/T}\otimes \mathcal{N})\oplus \bigwedge^2\mathcal{N},
		\]
		whose existence follows from Proposition \ref{prop:absoluteformsUVE}.(iv) and assertion (i). Proposition \ref{prop:relativeformsUVE}.(iv) then implies that the composite
		$
		d\mathcal{N} \rightarrow g_*\Omega^2_{A^\natural/T} \rightarrow g_*\Omega^2_{A^\natural/S}
		$
		is zero. Therefore, using Proposition \ref{prop:filtrationpushforward}.(ii), we get
		\[
		d\mathcal{N}\subset \ker(g_*\Omega^2_{A^\natural/T}\longrightarrow g_*\Omega^2_{A^\natural/S})=\operatorname{im}(\Omega^1_{S/T}\otimes g_*\Omega^1_{A^\natural/T}\stackrel{\wedge}\longrightarrow g_*\Omega^2_{A^\natural/T})\cong \Omega^2_{S/T}\oplus (\Omega^1_{S/T}\otimes \mathcal{N}).
		\]
		On the other hand, since pullbacks commute with the exterior derivative, given any section $\omega$ of $\mathcal{N}$, we have $e^\ast d\omega=0$. This shows that $d\omega$ is a section of $(\Omega^1_{S/T}\otimes \mathcal{N})\oplus \bigwedge^2\mathcal{N}$, under the above decomposition. Therefore,
		\[
		d\mathcal{N}\subset \left(\Omega^2_{S/T}\oplus (\Omega^1_{S/T}\otimes \mathcal{N})\right) \cap \left((\Omega^1_{S/T}\otimes \mathcal{N})\oplus \bigwedge^2\mathcal{N}\right)  \cong \operatorname{im}(\Omega^1_{S/T}\otimes \mathcal{N}\stackrel{\wedge}\longrightarrow g_*\Omega^2_{A^\natural/T}).\qedhere
		\]
	\end{proof}
	
	\begin{remark}
		Theorem \ref{thm:canonicalsplitting}.(i) generalizes to the case where $A^\natural$ is replaced by a finite (possibly empty) product $(A^\natural)^n\coloneqq A^\natural\times_S \cdots \times_S A^\natural$. Namely, denoting (by abuse) the structure morphism of $(A^\natural)^n$ by $g$ and its zero section by $e$, the morphism $\rho \coloneqq g_*(e^\ast): g_*\Omega^1_{(A^\natural)^n/T} \rightarrow \Omega^1_{S/T}$ is a retraction of the natural map $\Omega^1_{S/T}\hookrightarrow g_*\Omega^1_{(A^\natural)^n/T}$. The key point is that the analogue of Theorem \ref{thm:Laumon} holds for $(A^\natural)^n$ by the K\"unneth formula \cite[\href{https://stacks.math.columbia.edu/tag/0FLT}{Lemma 0FLT}]{Stacks}.
	\end{remark}
	
	Using Theorem \ref{thm:canonicalsplitting}, we can now prove the main result of this note.
	
	\begin{theorem} \label{thm:Dgroupscheme}
		Let $\mathcal{I} \coloneqq g^\ast\mathcal{N}$. The pair $(A^\natural,\mathcal{I})$ is a $D$-group scheme over $S$.
	\end{theorem}
	
	\begin{proof}
		By Proposition \ref{prop:absoluteformsUVE} and Theorem \ref{thm:canonicalsplitting}, the submodule $\mathcal{I} \hookrightarrow g^\ast g_*\Omega^1_{A^\natural/T}=\Omega^1_{A^\natural/T}$ is a splitting of \eqref{eqn:exactsequencedifferentials}. To prove integrability, note that Theorem \ref{thm:canonicalsplitting}.(ii) implies in particular that $d\mathcal{N} \subset \operatorname{im}(\mathcal{N}\otimes g_*\Omega^1_{A^\natural/T}\stackrel{\wedge}\rightarrow g_*\Omega^2_{A^\natural/T})$. Therefore, using the fact that the canonical map
		\[
		g^\ast\operatorname{im}(\mathcal{N}\otimes g_*\Omega^1_{A^\natural/T}\stackrel{\wedge}\longrightarrow g_*\Omega^2_{A^\natural/T}) \longrightarrow \operatorname{im}(\mathcal{I}\otimes \Omega^1_{A^\natural/T}\stackrel{\wedge}\longrightarrow \Omega^2_{A^\natural/T})
		\]
		is an isomorphism (which is clear by exactness of $g^\ast$ and Proposition \ref{prop:absoluteformsUVE}), we see that $\mathcal{I}$ is integrable.
		
		Now, let $m: A^\natural\times_S A^\natural\rightarrow A^\natural$ be the multiplication, and $h: A^\natural \times_S A^\natural\rightarrow S$ be the structure morphism. Pullback along $m$ induces a morphism $g_*(m^\ast): g_*\Omega^1_{A^\natural/T}\rightarrow h_*\Omega^1_{(A^\natural\times_S A^\natural)/T}$, which sends $\ker(g_*(e^\ast))$ into $\ker(h_*((e\times e)^\ast))$ (since $e: S\rightarrow A^\natural$ is a morphism of $S$-group schemes). It follows that $m$ is a morphism of $D$-schemes. That $e$ and the inversion map $i: A^\natural\rightarrow A^\natural$ are also morphisms of $D$-schemes is proved similarly.
	\end{proof}
	
	We end this subsection by giving a rather simple formula for the Gauss--Manin connection on the de Rham cohomology of $A/S$. Let $\sigma : g_*\Omega^1_{A^{\natural}/S} \to g_*\Omega^1_{A^{\natural}/T}$ be the splitting
	\begin{equation}\label{eqn:sigma-splitting}
		\begin{tikzcd}
			0 \arrow{r} & \Omega^1_{S/T} \arrow{r} & g_*\Omega^1_{A^{\natural}/T} \arrow{r} & g_*\Omega^1_{A^{\natural}/S} \arrow{r} \arrow[bend right]{l}[swap]{\sigma} & 0
		\end{tikzcd}
	\end{equation}
	corresponding to the retraction $\rho : g_*\Omega^1_{A^{\natural}/T} \to \Omega^1_{S/T}$. Concretely, if $\omega$ is a section of $g_*\Omega^1_{A^{\natural}/S}$ (a relative differential form), then $\sigma \omega$ is the unique lift of $\omega$ to a section of $g_*\Omega^1_{A^{\natural}/T}$ (an absolute differential form) which vanishes along the zero section of $A^{\natural}$.
	
	\begin{proposition} \label{prop:UVEGM}
		Under the identification $H^1_{\rm dR}(A/S)\cong g_*\Omega^1_{A^\natural/S}$ of Proposition \ref{prop:UVEdeRham}, the Gauss--Manin connection $\nabla: H^1_{\rm dR}(A/S)\rightarrow \Omega^1_{S/T}\otimes H^1_{\rm dR}(A/S)$ is given by
		\begin{equation} \label{eqn:connectinghomomorphism}
			\nabla \omega = d\sigma \omega  \mod \Omega^2_{S/T},
		\end{equation}
		where $d = d_{A^{\natural}/T}$ and $\Omega^1_{S/T}\otimes g_*\Omega^1_{A^{\natural}/S} \cong g_*\Omega^2_{A^{\natural}/T}/\Omega^2_{S/T}$ via Proposition \ref{prop:filtrationpushforward}.
	\end{proposition}
	
	\begin{proof}
		Consider the filtration $F^p = \mathrm{im}(g^*\Omega^p_{S/T}\otimes \Omega^{\bullet}_{A^{\natural}/T}[-p] \stackrel{\wedge}{\to} \Omega^{\bullet}_{A^{\natural}/T})$ of $\Omega^{\bullet}_{A^{\natural}/T}$ (cf. \S \ref{ssec:deRham}). By Proposition \ref{prop:relativeformsUVE}, the complexes $F^p/F^{p+1}\cong g^\ast\Omega^p_{S/T}\otimes \Omega^\bullet_{A^\natural/S}[-p]$  are $g_*$-acyclic, hence the differential $d_1^{0,1}: E_1^{0,1}\rightarrow E_1^{1,1}$ on the first page of the spectral sequence corresponding to $\{F^p\}_{p\ge 0}$ is given by the connecting homomorphism in the long exact sequence associated to
		\[
		\begin{tikzcd}
			0\arrow{r}&g_*(F^1/F^2)\arrow{r}\arrow[equal]{d}&g_*(F^0/F^2)\arrow{r}\arrow[equal]{d}&g_*(F^0/F^1)\arrow{r}\arrow{r}\arrow[equal]{d}&0\\
			0\arrow{r}&\Omega^1_{S/T}\otimes g_*\Omega^\bullet_{A^\natural/S}[-1]\arrow{r}&g_*(F^0/F^2)\arrow{r}&g_*\Omega^\bullet_{A^\natural/S}\arrow{r}\arrow{r}&0.
		\end{tikzcd}
		\]
		The desired assertion is now immediate.
	\end{proof}
	
	\begin{remark}
		Using that $H^q_{\rm dR}(A/S) \cong \bigwedge^q H^1_{\rm dR}(A/S)$, Proposition \ref{prop:UVEGM} also yields a similar formula for the Gauss--Manin connection on $H^q_{\rm dR}(A/S)$.
	\end{remark}
	
	\begin{remark} \label{rmk:GMsigma}
		Locally on $S$, we may choose a trivialization $\{\omega_i\}_{1\leq i\leq r}$ of $g_*\Omega^1_{A^\natural/S}$, and we may write $\nabla=d_{S/T}+M$, where $M=(\mu_{i,j})_{1\leq i,j\leq r} \in \operatorname{Mat}_{r\times r}(\Gamma(S,\Omega^1_{S/T}))$. Then, Theorem \eqref{thm:canonicalsplitting}.(ii) implies that
		\begin{equation} \label{eqn:differentialcanonicalsplitting}
			d\sigma\omega_j=\sum_{i=1}^r\mu_{i,j}\wedge\sigma\omega_i,
		\end{equation}
		in $g_*\Omega^2_{A^\natural/T}$, for all $1\leq j\leq r$. In other words, the lifts $\sigma\omega_j$ satisfy the Gauss--Manin equation on the nose, and not just modulo the submodule $\Omega^2_{S/T} \subset g_*\Omega^2_{A^\natural/T}$.
	\end{remark}
	
	\subsection{Comparison with the canonical analytic \texorpdfstring{$D$}{}-group scheme structure} \label{ssec:AnalyticDgroupscheme}
	
	Now, assume that $T=\operatorname{Spec} \mathbb C$, and consider the vector group $p: V \coloneqq \mathbb{V}(g_*\Omega^1_{A^{\natural}/S})\rightarrow S$ with the linear $D$-group scheme structure $\mathcal{I}_{\nabla} = \mathrm{im}(\sigma_{\nabla})$ induced by the Gauss--Manin connection $\nabla: g_*\Omega^1_{A^\natural/S}\rightarrow \Omega^1_{S/\mathbb C}\otimes g_*\Omega^1_{A^\natural/S}$ (under the identification $g_*\Omega^1_{A^{\natural}/S} \cong H^1_{\rm dR}(A/S)$), as explained in Example \ref{ex:linear-d-group-scheme}. By analytification, the pair $(V^{\rm an},\mathcal{I}^{\rm an}_{\nabla})$ is then an analytic $D$-group scheme over $S^{\rm an}$ (cf. \cite[6.2]{Bost}).
	
	The following proposition implies that $\mathcal{I}^{\rm an}_{\nabla}$ descends to $A^{\natural,\rm an}$ along the uniformization map $\exp: V^{\rm an}\rightarrow A^{\natural,\rm an}$ \eqref{eqn:uniformization}.
	\begin{proposition} \label{prop:descent}
		The $\mathcal{O}_{V^{\rm an}}$-submodule $\mathcal{I}^{\rm an}_{\nabla} \hookrightarrow \Omega^1_{V^{\rm an}}$ is $L$-invariant.
	\end{proposition}
	
	\begin{proof}
		The assertion is local on $S^{\rm an}$, so we may assume that $g_*\Omega^1_{A^\natural/S} \cong H^1_{\rm dR}(A/S)$ admits a trivialization $\{\omega_i\}_{1\leq i\leq r}$. As in Remark \ref{rmk:GMsigma}, we let $M=(\mu_{ij})_{1\leq i,j\leq r} \in \operatorname{Mat}_{r\times r}(\Gamma(S,\Omega^1_{S/\mathbb{C}}))$ be the connection matrix of $\nabla$. Let $z_1,\ldots,z_r: V^{\rm an}\rightarrow \mathbb{C}$ be the holomorphic coordinates dual to $\omega_1,\ldots,\omega_r$, so that $p^*\omega_j=dz_j$ under the canonical isomorphism $p^*g_*\Omega^1_{A^{\natural}/S} \cong \Omega^1_{V/S}$ (cf. Example \ref{ex:linear-d-group-scheme}). Then, $\mathcal{I}_{\nabla}^{\rm an}$ is generated by the 1-forms
		\[
		\sigma_\nabla dz_j=dz_j-\sum_{i=1}^r z_i\mu_{ij}\text{, }\qquad j=1,\ldots,r.
		\]
		The image of $L$ inside of $V^{\rm an}$ is the additive $S^{\rm an}$-subgroup spanned by the integration functionals $\int_\gamma$, for $\gamma$ a section of $(R^1f^{\rm an}_*\mathbb Z)^\vee$. The assertion now follows from
		\[
		d\left( \int_\gamma\omega_j \right)-\sum_{i=1}^r \left(\int_\gamma\omega_i\right)\mu_{ij}=0,
		\]
		which characterizes the Gauss--Manin connection.
	\end{proof}
	
	Let $\mathcal{J} \hookrightarrow \Omega^1_{A^{\natural,\rm an}}$ be the $\mathcal{O}_{A^{\natural,\rm an}}$-submodule obtained from $\mathcal{I}^{\rm an}_{\nabla}$ via its quotient by $L$. As remarked in \cite[6.4]{Bost}, a result of Grothendieck and Mazur--Messing \cite{MazurMessing} implies that the analytic vector bundle $\mathcal{J}$ is the analytification of an algebraic vector bundle. The next theorem gives in particular a new proof of this result.
	
	\begin{theorem} \label{thm:algebraic}
		We have an equality $\mathcal{I}^{\rm an}=\mathcal{J}$ of $\mathcal{O}_{A^{\natural,\rm an}}$-submodules of $\Omega^1_{A^{\natural,\rm an}}$.
	\end{theorem}
	
	\begin{proof}
		We keep the notation of the proof of Proposition \ref{prop:descent}. The assertion is local on $S$, so we may assume that $g_*\Omega^1_{A^\natural/S} \cong H^1_{\rm dR}(A/S)$ admits a trivialization $\{\omega_i\}_{1\leq i\leq r}$. Then $\mathcal{J}$ is trivialized by the $L$-invariant sections $\{\sigma_{\nabla}dz_j\}_{1\leq i\leq r}$, whereas $\mathcal{I}$ is trivialized by $\{\sigma\omega_i\}_{1\leq i\leq r}$. Therefore, it is enough to prove that
		\[
		\sigma\omega_j=\sigma_{\nabla}dz_j=dz_j-\sum_{i=1}^rz_i\mu_{ij},
		\]
		as sections of $\Omega^1_{A^{\natural,\rm an}}$, for all $1\leq j\leq r$.
		
		Since $\sigma \omega_j$ and $\sigma_{\nabla}dz_j$ both project to $\omega_j$ in $\Omega^1_{A^{\natural,\rm an}/S^{\rm an}}$, there exists $\alpha_j \in \Gamma(A^{\natural,\rm an},(g^{\rm an})^*\Omega^1_{S^{\rm an}})$ such that
		\begin{equation} \label{eqn:algebraic}
			\sigma \omega_j=dz_j-\sum_{i=1}^rz_i\mu_{ij}+\alpha_j.
		\end{equation}
		Plugging this equation into \eqref{eqn:differentialcanonicalsplitting} and using integrability of the Gauss--Manin connection, namely $dM+M\wedge M=0$, we obtain
		\begin{equation} \label{eqn:algebraic3}
			d\alpha_j=\sum_{i=1}^r\mu_{ij}\wedge\alpha_i.
		\end{equation}
		Now, by working locally on $S^{\rm an}$, we may assume that $\Omega^1_{S^{\rm an}}$ is trivialized by $ds_1,\ldots,ds_m$, for analytic coordinates $s_1,\ldots,s_m: S^{\rm an}\rightarrow \mathbb C$. Writing $\alpha_j=\sum_{i=1}^m\varphi_{ij}ds_i$, equation \eqref{eqn:algebraic3} yields that $\partial \varphi_{ij}/\partial z_k =0$, for all $1\leq k\leq r$. This implies that $\alpha_j$ only depends on the variables $s_i$, hence is the pullback to $A^{\natural,\rm an}$ of a one-form on $S^{\rm an}$. In particular, we have $\rho(\alpha_j)=\alpha_j$, where $\rho$ is the retraction given by Theorem \ref{thm:canonicalsplitting}.(i). On the other hand, by applying $\rho$ to both sides of \eqref{eqn:algebraic} and using that the coordinates $z_i$ all vanish upon restriction to the zero section, we see that $\alpha_j=0$, as we wanted.
	\end{proof}

\end{document}